\documentclass[11pt, notitlepage]{article}
\usepackage{amssymb,amsmath,comment}
\catcode`\@=11 \@addtoreset{equation}{section}
\def\thesection{\arabic{section}}

\def\theequation{\thesection.\arabic{equation}}
\catcode`\@=12
\usepackage{colortbl}
\usepackage{a4wide}

\newcommand{\ds} {\displaystyle}
\newcommand{\e}{\varepsilon}
\newcommand{\pa} {\partial}
\newcommand{\al} {\alpha}

\newcommand{\de} {\delta}

\newcommand{\Om} {\Omega}
\newcommand{\ra} {\rightarrow}
\newcommand{\rp} {\rightharpoonup}
\newcommand{\De} {\Delta}
\newcommand{\la} {\lambda}
\newcommand{\La} {\Lambda}
\newcommand{\noi} {\noindent}
\newcommand{\na} {\nabla}

\newcommand{\ld} {\langle}
\newcommand{\rd} {\rangle}

\usepackage[all]{xy}
\catcode`\@=11
\def\theequation{\@arabic{\c@section}.\@arabic{\c@equation}}
\catcode`\@=12

\def\QED{\hfill {$\square$}\goodbreak \medskip}

\newtheorem{Theorem}{Theorem}[section]
\newtheorem{Lemma}[Theorem]{Lemma}
\newtheorem{Proposition}[Theorem]{Proposition}

\newtheorem{Definition}[Theorem]{Definition}

\begin{document}
\vspace{0.01in}

\title
{
	The effect of topology  on the number of positive solutions of  elliptic equation  involving Hardy-Littlewood-Sobolev critical exponent  }

\author{ {\bf Divya Goel\footnote{e-mail: divyagoel2511@gmail.com}} \\ Department of Mathematics,\\ Indian Institute of Technology Delhi,\\
	Hauz Khaz, New Delhi-110016, India. }

\date{}

\maketitle

\begin{abstract}
In this article we are concern for the  following Choquard equation
\[
-\Delta u =  \la|u|^{q-2}u +\left(\int_{\Omega}\frac{|u(y)|^{2^*_{\mu}}}{|x-y|^{\mu}}dy\right)|u|^{2^*_{\mu}-2}u  \; \text{in}\;
\Omega,\quad 
 u = 0 \;    \text{ on } \partial \Omega ,
\]
where $\Omega$ is an open  bounded set  with continuous boundary in $\mathbb{R}^N( N\geq 3)$,  $2^*_{\mu}=\frac{2N-\mu}{N-2}$ and $q \in [2,2^*)$ where $2^*=\frac{2N}{N-2}$. Using Lusternik-Schnirelman theory, we associate  the  number of positive solutions of the above problem with the topology of $\Omega$.  Indeed, we prove if $\la< \la_1$ then  problem has  $\text{cat}_{\Omega}(\Omega)$ positive solutions whenever 
$q \in [2,2^*)$ and $N>3  $ or $4<q<6 $ and $N=3$.
\medskip
  
\noi \textbf{Key words:} Choquard equation, critical exponent, Lusternik-Schnirelman theory. 

\medskip

\noi \textit{2010 Mathematics Subject Classification: 49J35, 35A15, 35J60.} 

\end{abstract}
\section{Introduction}

The purpose of this article is to study the existence and multiplicity of solution  of the following Choquard equation 
\begin{equation*}
(P_\la)\;
\left\{\begin{array}{rllll}
-\De u
&=\la|u|^{q-2}u+ \left(\ds \int_{\Om}\frac{|u(y)|^{2^*_{\mu}}}{|x-y|^{\mu}}dy\right)|u|^{2^*_{\mu}-2}u  \; \text{in}\;
\Om,\\
u&=0 \; \text{ on } \pa \Om,
\end{array}
\right.
\end{equation*}
where $\Om$ is an open  bounded set  with continuous boundary in $\mathbb{R}^N( N\geq 3)$,  $2^*_{\mu}=\frac{2N-\mu}{N-2}$ and $q \in [2,2^*)$ where $2^*=\frac{2N}{N-2}$.\\

It is not unfamiliar that nonlinear analysis fascinates many researchers. In particular, the study of elliptic equations is more attractive both for theoretical pde's and real-world applications. 
There is an ample amount of literature regarding the existence and multiplicity of solutions of the following equation: 
\begin{equation}\label{lu13}
-\De u = \la |u|^{q-2}u+ |u|^{2^*-2}u  \text{ in } \Om, \quad u=0 \text{ on } \pa \Om. 
\end{equation}

\noi In the pioneering work of Brezis and Nirenberg \cite{brezis}, authors studied the problem \eqref{lu13} with $q=2$ for the existence of a nontrivial solution. Then many researchers studied the elliptic equations involving Sobolev critical exponent in bounded and unbounded domains.   In \cite{bahri},  Bahri and Coron studied the problem \eqref{lu13} in case of  $\la=0$  and proved the existence of a positive solution when $\Om$ is not a contractible domain using homology theory. Subsequently,  Rey  \cite{rey}
studied  critical elliptic problem  \eqref{lu13} for $q=2$ and proved   that  there exist at least $\text{cat}_{\Om}(\Om)$ solutions in $H_0^1(\Om)$ whenever $\la$ is sufficiently small.  We cite  \cite{benci1, benci, dancer, alves, yin} for existence and multiplicity of solutions of elliptic problems using variational methods, with no attempt to provide the complete list. In the framework of fractional Laplacian,  the effect of topology on the number of solutions of problems was discussed in \cite{figueiredo, figueiredo1} and references therein. \\
\noi Currently, nonlocal equations  appealed a  substantial number of researchers, especially the  Choquard equations.  The work on Choquard equations was started with the  quantum theory of a polaron model given by S. Pekar \cite{pekar} in 1954.  After that in 1976, in the modeling of a one component plasma,  P. Choquard \cite{ehleib} used the following equation with $\mu=1,\; p=2$ and $N=3$:
\begin{equation}\label{lu12}
-\De u +u = \left(\frac{1}{|x|^{\mu}}* |u|^p\right) |u|^{p-2}u \text{ in } \mathbb{R}^N.
\end{equation}
For $\mu=1,\; p=2$ and $N=3$,   Lieb \cite{ehleib}   proved existence, uniqueness of the ground state solution of \eqref{lu12} by using symmetric decreasing rearrangement inequalities. With the help of variational methods, Moroz and Schaftingen \cite{moroz2}  established the existence of  least energy solutions of \eqref{lu12} and prove   properties about the symmetry, regularity, and asymptotic behavior at infinity of the least energy solutions. For interested readers, we refer \cite{alves1, clapp1, clapp,  moroz3}  and references therein for the work on Choquard equations. \\ 

\noi The Hardy-Littlewood-Sobolev inequality  \eqref{co9} plays a significant role in the variational formulation of  Choquard equations. Observe that the integral 
\begin{align*}
\int_{\Om}\int_{\Om}\frac{|u(x)|^{q}|u(y)|^{q}}{|x-y|^{\mu}}~dydx
\end{align*}
is well defined if $\frac{2N-\mu}{N}\leq q \leq \frac{2N-\mu}{N-2}=2^*_{\mu} $. Choquard equations  involving Hardy-Littlewood-Sobolev critical exponent(that is, $q = 2^*_{\mu}$) provoke the interest of the mathematical community due to the  lack of compactness in the embedding  $H_0^1(\Om) \ni u \mapsto \frac{|u|^{2^*_{\mu}}|u|^{2^*_{\mu}}}{|x-y|^{\mu}}   \in  L^{1}(\Om\times \Om)$. In  \cite{yangjmaa},  authors used variational methods to prove the existence and multiplicity  of positive solutions for the  critical Choquard problem involving convex and convex-concave type nonlinearities.\\

\noi In this spirit, recently in \cite{choqcoron} Goel, R\u adulescu and Sreenadh,  studied the Coron problem for Choquard equation and  proved the existence of a positive  high energy solution of the following problem  
\begin{equation*}
-\De u = \left(\int_{\Om}\frac{|u(y)|^{2^*_{\mu}}}{|x-y|^{\mu}}dy\right)|u|^{2^*_{\mu}-2}u   \; \text{in}\;
\Om,\quad 
u = 0 \;    \text{ on } \pa \Om,
\end{equation*}
where $\Om$ is a smooth  bounded domain in $\mathbb{R}^N( N\geq 3)$,  $2^*_{\mu}=\frac{2N-\mu}{N-2}$,  $0< \mu <N$ and 
satisfies the following conditions:
There exists constants $0<R_1<R_2<\infty$ such that 
\begin{align*}
\{ x \in \mathbb{R}^N ,\; R_1<|x|<R_2 \} \subset\Om, \qquad
\{ x \in \mathbb{R}^N ,\; |x|<R_1 \} \nsubseteq \overline{\Om}.
\end{align*} 
In \cite{ghimenti} Ghimenti and Pagliardini studied the following slightly subcritical Choquard problem
\begin{equation}\label{lu17}
-\De u -\la u  = \left(\int_{\Om}\frac{|u(y)|^{p_\e}}{|x-y|^{\mu}}dy\right)|u|^{p_\e-2}u   \; \text{in}\;
\Om,\quad 
u = 0 \;    \text{ on } \pa \Om,
\end{equation}
where $\e>0,\; \Om$ is a regular bounded domain of $\mathbb{R}^N,\; \la\geq0$ and $p_\e= 2^*_\mu -\e$. 
Here authors proved that There exists $\overline{\e} >0$  such that for every $\e \in (0, \overline{\e}]$, Problem \eqref{lu17} has at least $cat_{\Om}(\Om)$ low energy solutions. Moreover, if $\Om$ is not
contractible, there exists another solution with higher energy. \\
Motivated by all these, in this paper, we study the existence of multiple solutions of the problem $(P_\la)$.  Since the geometry of the domain plays an essential role, here  we  proved that the  topology of the domain yields  a lower bound on the number of positive solutions. More precisely, we show that  the problem $(P_\la)$ has at least  $\text{cat}_{\Om}(\Om)$ solutions. Here $\text{cat}_{\Om}(\Om)$ is the  Lusternik-Schirelman category defined as follows
\begin{Definition}
	Let $X$ be  a topological space and $Y$ be a closed set in X Then 
	\begin{align*}
	Cat_X(Y)= \min \bigg\{  k & \in \mathbb{N} \; \bigg|  \text{ there exists closed subsets } Y_1,Y_2,\cdots Y_k \subset X \text{ such that } \\
	& Y_j \text{ is contractible to a point in } X \text{ for all } j \text{ and } \ds \ds \cup_{j=1}^{k} Y_j =X     \bigg\}
	\end{align*}
\end{Definition}
In order to achieve our aim, we used the fact that  Lusternik-Schirelman category is invariant under Nehari manifold. Then using the blowup analysis involving the minimizers and the mountain pass Lemma, we show  the infimum of the functional associated with $(P_\la)$ over the the Nehari Manifold is achieved. Moreover we define the barycenter mapping associated to Choquard nonlinear term and apply the machinery of barycenter mapping to prove our desired conclusion. With this introduction we will state our main result: 

\begin{Theorem}\label{luthm1}
	Let  $\Om$ is an open  bounded set  with continuous boundary in $\mathbb{R}^N( N\geq 3)$ and $q \in [2,2^*)$ then there exists $0<\La^*<\la_1$ such that for all $\la \in (0,\La^*)$ there exists at least 
	$cat_{\Om}(\Om)$ positive solutions of $(P_\la)$ under the following conditions 
	\begin{enumerate}
		\item $q \in [2,2^*)$ and $N> 3$ or 
		\item  $4<q<6 $ and $N=3$.
	\end{enumerate}
\end{Theorem}

Turning to layout of the article: In Section 2, we give the variational framework and preliminary results. In Section 3, we give the Palais-Smale analysis and existence of a solution of $(P_\la)$. In Section 4, we prove some technical Lemmas and proof  Theorem \ref{luthm1}.  Finally in the appendix, we study  the behavior of  optimizing sequence  of the best constant  $S_{H,L}$ defined  in \eqref{nh5}.

\section{Variational framework and Preliminary results}
To study the problem $(P_\la)$ by variational approach we will start with the stating the celebrated Hardy-Littlewood-Sobolev inequality. 
\begin{Proposition}\cite{leib}(\textbf{Hardy-Littlewood-Sobolev Inequality})
	Let $t,r>1$ and $0<\mu <N$ with $1/t+\mu/N+1/r=2$, $f\in L^t(\mathbb{R}^N)$ and $h\in L^r(\mathbb{R}^N)$. There exists a sharp constant $C(t,r,\mu,N)$ independent of $f,h$, such that 
	\begin{equation}\label{co9}
	\int_{\mathbb{R}^N}\int_{\mathbb{R}^N}\frac{f(x)h(y)}{|x-y|^{\mu}}~dydx \leq C(t,r,\mu,N) \|f\|_{L^t}\|h\|_{L^r}.
	\end{equation}
	If $t=r=2N/(2N-\mu)$, then 
	\begin{align*}
	C(t,r,\mu,N)=C(N,\mu)= \pi^{\frac{\mu}{2}}\frac{\Gamma(\frac{N}{2}-\frac{\mu}{2})}{\Gamma(N-\frac{\mu}{2})}\left\lbrace \frac{\Gamma(\frac{N}{2})}{\Gamma(\frac{\mu}{2})}\right\rbrace^{-1+\frac{\mu}{N}}.
	\end{align*}
	Equality holds in  \eqref{co9} if and only if $f\equiv (constant)h$ and 
	\begin{align*}
	h(x)= A(\gamma^2+|x-a|^2)^{(2N-\mu)/2},
	\end{align*} 
	for some $A\in \mathbb{C}, 0\neq \gamma \in \mathbb{R}$ and $a \in \mathbb{R}^N$. \QED
\end{Proposition}
The Sobolev space $D^{1,2}(\mathbb{R}^N)$ is defined as 
\begin{align*}
D^{1,2}(\mathbb{R}^N)= \bigg\{ u \in L^{2^*}(\mathbb{R}^N)\;:\; \na u \in L^2(\mathbb{R}^N,\mathbb{R}^N)   \bigg\},
\end{align*}
endowed with the norm 
\begin{align*}
\|u\|= \left(\int_{\mathbb{R}^N} |\na u|^2~dx\right)^{\frac{1}{2}}. 
\end{align*}

\noi The best constant for the embedding $D^{1,2}(\mathbb{R}^N)$ into  $L^{2^*}(\mathbb{R}^N)$ (where  $2^*= \frac{2N}{N-2}$ )is defined as 
\begin{align*}
S=\inf_{ u \in D^{1,2}(\mathbb{R}^N) \setminus \{0\}} \left\{ \int_{\mathbb{R}^N} |\na u|^2 dx:\; \int_{\mathbb{R}^N}|u|^{2^{*}}dx=1\right\}.
\end{align*}
Consequently, we define 
\begin{equation}\label{nh5}
S_{H,L}= \inf_{ u \in D^{1,2}(\mathbb{R}^N)\setminus \{0\}} \left\{  \int_{\mathbb{R}^N}|\nabla u|^2 dx:\;  \int_{\mathbb{R}^{N}} \int_{\mathbb{R}^N } \frac{|u(x)|^{2^{*}_{\mu}}|u(y)|^{2^{*}_{\mu}}}{|x-y|^{\mu}}~dxdy=1 \right\}.
\end{equation}

\begin{Lemma}\label{lulem13}
	\cite{yang}
	The constant $S_{H,L}$ defined in \eqref{nh5} is achieved if  and only if 
	\begin{align*}
	u=C\left(\frac{b}{b^2+|x-a|^2}\right)^{\frac{N-2}{2}}
	\end{align*} 
	where $C>0$ is a fixed constant , $a\in \mathbb{R}^N$ and $b\in (0,\infty)$ are parameters. Moreover,
	\begin{align*}
	S=	S_{H,L} \left(C(N,\mu)\right)^{\frac{N-2}{2N -\mu}}.
	\end{align*}
\end{Lemma}
\begin{Lemma}\cite{yang}
	For $N\geq 3$ and $0<\mu<N$. Then 
	\begin{align*}
	\|.\|_{NL}:= \left(\ds \int_{\Om}\int_{\Om}\frac{|.|^{2^{*}_{\mu}}|.|^{2^{*}_{\mu}}}{|x-y|^{\mu}}~dydx\right)^{\frac{1}{2\cdot 2^{*}_{\mu}}}
	\end{align*}
	defines a norm on $L^{2^*}(\Om)$, where $\Om$ is an open  bounded set  with continuous boundary in $\mathbb{R}^N$.
\end{Lemma} 
\noi The energy functional associated with $(P_\la)$,  $J_\la: H_0^1(\Om) \ra \mathbb{R}$ is defined by  
\begin{equation*} 
J_\la(u)=\frac{1}{2}\int_{\Om} |\na u|^2~dx -\frac{\la}{q}\int_{\Om}|u|^q~dx- \frac{1}{2\cdot 2^{*}_{\mu}}\int_{\Om}\int_{\Om}\frac{|u(x)|^{2^{*}_{\mu}}|u(y)|^{2^{*}_{\mu}}}{|x-y|^{\mu}}~dydx. 
\end{equation*}
Employing the Hardy-Littlewood-Sobolev inequality \eqref{co9}, we have 
\begin{align*}
\left(\ds \int_{\Om}\int_{\Om}\frac{|u(x)|^{2^{*}_{\mu}}|u(y)|^{2^{*}_{\mu}}}{|x-y|^{\mu}}~dydx\right)^{\frac{1}{2^{*}_{\mu}}}\leq C(N,\mu)^{\frac{2N-\mu}{N-2}}\|u\|_{L^{2^*}}^2. 
\end{align*}
It implies the functional $J_\la \in C^1(H_0^{1}(\Om),\mathbb{R})$.  We know that there exists a one to one correspondence between the critical points of $J_\la$ and solution of $(P_\la)$.\\

\textbf{Notation}
	We  denote  $\la_1$  be the first eigenvalue of $-\De$ with zero Dirichlet boundary data, which is given by 
	\begin{align*}
	\la_1= \inf_{ u \in H_0^1(\Om)\setminus \{0\}}  \bigg \{   \int_{\Om}|\nabla u|^2~dx \bigg| \int_{\Om}| u|^2~dx=1 \bigg\}.
	\end{align*}
We also denote $\mathcal{(Q)}$ as  the  following condition:
	\begin{align*}
	\mathcal{(Q)} \text{ Assume } 0<\la<\la_1. \text{ Moreover, }  q \in [2,2^*) \text{ and }  N> 3 \;\;\text{   OR   } \;\;  4< q<6  \text{ and } N=3.
	\end{align*}
\begin{Lemma}\label{efflem1}
Assume   $N\geq 3$   and 	$   \la \in (0,\la_1)$.  Then	 $J_\la$ satisfies the following conditions:
	\begin{enumerate}
		\item [(i)] There exists $\al,\rho>0$ such that $J_\la(u)\geq \al$ for $\|u\|= \rho$ 	
		\item [(ii)] There exists  $e \in H_0^1(\Om)$ with $\|e\|>\rho$ such that $J_\la(e)<0$.
	\end{enumerate}
\end{Lemma}
\begin{proof}
	(i) Using H\"older's inequality, Sobolev inequality and Hardy-Littlewood-Sobolev inequality, we have 
	\begin{align*}
	J_\la(u)\geq 
	  \left\{
	\begin{array}{ll}
	\frac{1}{2}\left(1-\frac{\la}{\la_1}\right)\|u\|^2-\frac{S_{H,L}^{-1}}{2\cdot 2^*_{\mu}} \|u\|^{2\cdot 2^*_{\mu}}, & \text{ if }  q=2,\\
	\frac{1}{2}\|u\|^2- \frac{\la S^{\frac{-q}{2}} |\Om|^{\frac{2^*-q}{2^*}}}{q}\|u\|^q-\frac{S_{H,L}^{-1}}{2\cdot 2^*_{\mu}} \|u\|^{2\cdot 2^*_{\mu}}, &  \text{ if }  q\in (2,2^*) . \\
	\end{array} 
	\right.
	\end{align*}
	Using the given assumption on $\la$ and the fact that  $2<2\cdot 2^*_{\mu}$, we can choose $\al,\rho>0$ such that $J_\la(u)\geq \al$ whenever $\|u\|= \rho$. \\
	(ii) Let $u \in H_0^1(\Om)$ then 
	\begin{align*}
	J_\la(tu)& = \frac{ t^2}{2}\|u\|^2  - \frac{t^q}{q} \int_{\Om}|u|^q ~dx- \frac{t^{2\cdot 2^*_{\mu}}}{2\cdot 2^*_{\mu}} \int_{\Om}\int_{\Om} \frac{|u(x)|^{2^{*}_{\mu}}|u(y)|^{2^{*}_{\mu}}}{|x-y|^{\mu}}~dxdy \ra -\infty \text{ as } t \ra \infty.
	\end{align*}
	Hence we can choose $t_0>0$ such that $e:= t_0u$ such that (ii) follows. \QED
\end{proof}
The Nehari manifold associated to $J_\la$ defined as 
\begin{align*}
N_\la^{\Om}:=\{ u \in H_0^1(\Om)\setminus \{0\} \;|\;  \ld J_\la^{\prime}(u),u\rd=0 \}.
\end{align*}
\begin{Lemma}\label{lulem12}
	Let $u$ be a critical point on $N_\la^{\Om}$.  Then $u$ is a critical point of $J_\la$ on $H_0^1(\Om)$. 
\end{Lemma}
\begin{proof}
	  The proof follows from  \cite{dpp}. \QED
	\end{proof}
\begin{Lemma}\label{lulem8} 
	 Assume   	$    \la \in (0,\la_1)$.  Then $	N_\la^\Om \not = \emptyset$ and $J_\la$ is bounded below on $N_\la^\Om$.
\end{Lemma}
\begin{proof}
	Let $u \in H_0^1(\Om)\setminus \{0\}$.	Consider the function
	\begin{align*}
	\mathcal{\phi}_u(t)= &   J_\la(tu) = \frac{t^2}{2} \|u\|^2 -\frac{\la t^q}{q} \int_{\Om}|u|^q~dx- \frac{t^{2\cdot 2^{*}_{\mu}}}{2\cdot 2^{*}_{\mu}}\|u\|_{NL}^{2\cdot 2^{*}_{\mu}}. 
		\end{align*}
 Then $\phi_u(t)=0,\; \phi_u(t)\ra -\infty$ as $t \ra \infty$. We now show that there exists  unique $t_0>0$  such that $\phi_u^\prime(t_0)=0$ . 
 Since 
 \begin{align*}
 \phi_u^\prime(t)= t\|u\|^2- \la t^{q-1}\int_{\Om}|u|^q~dx-t^{2\cdot 2^{*}_{\mu}-1}\|u\|_{NL}^{2\cdot 2^{*}_{\mu}}= t m_u(t)
 \end{align*}
 	where $m_u(t)= \|u\|^2 -b_u(t)$ and $b_u(t)= \la t^{q-2} \ds\int_{\Om}|u|^q~dx+ t^{2\cdot 2^{*}_{\mu}-2}\|u\|_{NL}^{2\cdot 2^{*}_{\mu}}$.  Observe that $b_u$ is a continuous function, $\ds \lim_{t \ra \infty}b_u(t) = \infty$ and $b_u^\prime(t) >0$ for all $t>0$. Therefore, there exists unique $t_0>0$ such that $b_u(t_0)= \|u\|^2$.  That is, $\phi_u^\prime(t_0)=0$. 
 	It implies $t_0 \phi_u^\prime(t_0)=0$ and $t_0u \in N_\la^{\Om}$.  		It implies $N_\la^{\Om} \not = \emptyset$. Now if $ u \in N_\la^\Om$ then $J_\la(u)$ reduced to 
	\begin{align*}
	J_\la(u)= \left(\frac{1}{2}-\frac{1}{q}\right)\int_{\Om}|u|^q~dx+ \left(\frac{1}{2}-\frac{1}{2\cdot 2^{*}_{\mu}}\right)\|u\|_{NL}^{2\cdot 2^{*}_{\mu}}>0 .
		\end{align*}
Therefore, $\ds \inf_{u \in N_\la^{\Om}} J_\la(u)>0$. That is,  $J_\la$ is bounded below on $N_\la^\Om$.
	\QED
	\end{proof}
Now we set 
\begin{align}\label{lu14}
\theta_\la:= \inf_{u \in N_\la^{\Om}} J_\la(u)\quad \text{and} \quad \widehat{\theta}_\la:= \inf_{ u \in H_0^1(\Om)\setminus \{0\}} \sup_{t\geq 0}J_\la(tu),
\end{align}
where $\widehat{\theta}_\la$ denote the Mountain Pass (MP, in short) level.
\section{The Palais-Smale condition and estimates of the functional }
In this section we will give the  Palais--Smale analysis and prove the  existence of a minimizer of the functional $J_\la$ over the Nehari manifold. 

\begin{Lemma}\label{lulem2}
Let $N\geq 3,\; \la\in (0,\la_1)$ and $q \in [2,2^*)$. Then the functional $J_\la$ satisfies the $(PS)_c$ condition for all $c<\frac{N-\mu+2}{2(2N-\mu)}S_{H,L}^{\frac{2N-\mu}{N-\mu+2}}$. 
\end{Lemma}
\begin{proof}
	Let $\{u_n\}$ be a sequence in $H_0^1(\Om)$ such that 
	\begin{align}\label{lu4}
	J_\la(u_n)\ra c \text{ and }  \left\ld J_\la^{\prime}(u_n), \frac{u_n}{\|u_n\|}\right\rd  \ra 0 \text{ as } n \ra \infty.
	\end{align}
	\textbf{ Claim 1: } $u_n$ is a bounded sequence in $H_0^1(\Om)$.\\

	On the contrary assume that $\|u_n\| \ra \infty$. Let $\widetilde{u_n}= \ds \frac{u_n}{\|u_n\|}$ be a sequence in $H_0^1(\Om)$ then $\|\widetilde{u_n}\| =1$ for all $n$. Therefore we can assume there exists $\widetilde{u}$, up to   subsequences
	\begin{align*}
	\widetilde{u_n} \rp \widetilde{u}  \text{ weakly in } H_0^1(\Om), \quad \widetilde{u_n} \ra \widetilde{u}  \text{ strongly in } L^r(\Om) \text{ for all } r \in [1,2^*).
	\end{align*}
	Using \eqref{lu4} we have 
	\begin{align*}
	& \frac{1}{2}\|\widetilde{u_n}\|^2-\frac{\la}{q} \|u_n\|^{q-2}\int_{\Om}|\widetilde{u_n}|^q~dx- \frac{1}{2\cdot 2^*_{\mu}}\|u_n\|^{2\cdot 2^*_{\mu}-2}\|\widetilde{u_n}\|_{NL}^{2\cdot 2^*_{\mu}}= o_n(1) \quad \text{ and }\\
	& \|\widetilde{u_n}\|^2- \la \|u_n\|^{q-2}\int_{\Om}|\widetilde{u_n}|^q~dx- \|u_n\|^{2\cdot 2^*_{\mu}-2}\|\widetilde{u_n}\|_{NL}^{2\cdot 2^*_{\mu}}= o_n(1). 
	\end{align*}
	It implies that 
	\begin{align*}
	\left(\frac{1}{2}-\frac{1}{2\cdot 2^*_{\mu}} \right)\|\widetilde{u_n}\|^2=  \left(\frac{1}{q}- \frac{1}{2\cdot 2^*_{\mu}}\right) \la \|u_n\|^{q-2}\int_{\Om}|\widetilde{u_n}|^q~dx+ o_n(1).
	\end{align*}
	Now if $q>2$ and $\la>0$ then by the assumption $\|u_n\|\ra \infty$, we get $\|\widetilde{u_n}\| \ra \infty$, which is not possible. If $q=2$ and $\la \in (0,\la_1)$,  then $0< \left(1-\frac{\la}{\la_1}\right)\|u_n\|^{2}\leq  o_n(1)$, which is again not possible, this concludes the proof of Claim. \\
 Hence we can assume,  there exists a $u_0 \in H_0^1(\Om)$ such that  up to a subsequence $u_n\rp u_0$ weakly in $H_0^1(\Om)$,   $u_n \ra u_0  \text{ strongly in } L^r(\Om) \text{ for all } r \in [1,2^*)$ and $u_n \ra u_0 $ a.e. on $\Om$.  Using all this and  proceeding with the same assertions as in \cite[Lemma 2.4]{yang}, we get $J_\la^{\prime}(u_0)=0$. Now  the Brezis-Leib Lemma (See \cite{bre1983, yang}) leads to 
\begin{align*}
J_\la(u_n)
& = J_\la(u_0)+ \frac{1}{2}\|u_n-u_0\|^2- \frac{1}{2\cdot 2^*_{\mu}}\|u_n-u_0\|_{NL}^{2\cdot 2^*_{\mu}} +o_n(1)
\end{align*}
and  
 \begin{align}
o_n(1)& = \ld J_\la^{\prime}(u_n)-J_\la^{\prime}(u_0) ,\; u_n- u_0 \rd \nonumber\\
& = \|u_n\|^2 - \|u_0\|^2 - \|u_n\|_{NL}^{2\cdot 2^*_{\mu}}+ \|u_0\|_{NL}^{2\cdot 2^*_{\mu}} = \|u_n-u_0\|^2  - \|u_n-u_0\|_{NL}^{2\cdot 2^*_{\mu}}\label{lu15}.
 \end{align}
It implies  $J_\la(u_0)+ \frac{N-\mu+2}{2(2N-\mu)}\|u_n-u_0\|^2 =c+o_n(1)$ and if  $\|u_n-u_0\|^2 \ra M$  as $n\ra \infty$ then by \eqref{lu15},  $\|u_n-u_0\|_{NL}^{2\cdot 2^*_{\mu}}  \ra M$ as $n\ra \infty$. If $M=0$ then we are done otherwise  if $M>0$ then using the definition of $S_{H,L}$, we have $M^{\frac{1}{2^*_{\mu}}}S_{H,L}\leq M$ that is, $S_{H,L}^{\frac{2N-\mu}{N-\mu+2}} \leq M$. Since $\ld J_\la^{\prime}(u_0),\; u_0 \rd =0 $, it gives 
\begin{align*}
J_\la(u_0)= \left(\frac{1}{2}-\frac{1}{q}\right)\|u_0\|^2+ \left(\frac{1}{2}-\frac{1}{2\cdot 2^*_{\mu}}\right)\|u_0\|_{NL}^{2\cdot 2^*_{\mu}} \geq 0. 
\end{align*}
Resuming the information collected so far, what we have gained is that, 
\begin{align*}
o_n(1)+c= J_\la(u_0)+ \frac{N-\mu+2}{2(2N-\mu)}M \geq \frac{N-\mu+2}{2(2N-\mu)}S_{H,L}^{\frac{2N-\mu}{N-\mu+2}}, 
\end{align*}
which yields a contradiction to the range of $c$. Hence compactness of the sequence follows. \QED
	\end{proof}
\begin{Lemma}\label{lulem14}
	Let $N\geq 3$ and $\la\in (0,\la_1)$ then $J_\la$ constraint to $N_\la^{\Om}$ satisfies the $(PS)_c$ condition for all $c<\frac{N-\mu+2}{2(2N-\mu)}S_{H,L}^{\frac{2N-\mu}{N-\mu+2}}$. 
\end{Lemma}
\begin{proof}
	Let $u_n\in N_\la^{\Om}$ be such that $J_\la(u_n) \ra c$ and there exists a sequence $\{\al_n \}$ in $\mathbb{R}$ with 
	\begin{align}\label{lu20}
	\sup \{|\ld J_\la^\prime(u_n)-  \al_n T_\la^{\prime}(u_n) , \phi \rd | : \phi \in H_0^1(\Om), \|\phi \|=1 \} \ra 0 \text{ as } n\ra \infty, 
	\end{align}
	where the functional $T_\la$ is defined as $T_\la(u)= \|u\|^2 -\la \int_{ \Om}|u|^q~dx - \|u\|_{NL}^{2\cdot 2^*_{\mu}}$. 	First of all, we will show that $u_n$ is a bounded sequence in $H_0^1(\Om)$.	From the fact that $J_\la(u_n) \ra c$, it is easy to see that there exists a positive constant $C_1$ such that $
	|J_\la(u_n)|<C_1$. 
	If $q \in (2,2^*)$ then using the fact that $u_n \in N_\la^{\Om}$, we deduce that 
	\begin{align*}
	C_1& > J_\la(u_n)- \frac{1}{q} \ld J_\la^{\prime}(u_n), u_n \rd\\
	& =  \left( \frac{1}{2}-\frac{1}{q}\right) \|u_n\|^2 + \left( \frac{1}{q}-\frac{1}{2\cdot 2^*_{\mu}}\right) \|u_n\|_{NL}^{2\cdot 2^*_\mu}\\
	& \geq  \left( \frac{1}{2}-\frac{1}{q}\right) \|u_n\|^2. 
	\end{align*}
	If $q=2$, for $\la \in (0,\la_1)$, we obtain, for any $n \in \mathbb{N}$, 
	\begin{align*}
	C_1
	&> J_\la(u_n)- \frac{1}{2\cdot 2^*_{\mu}} \ld J_\la^{\prime}(u_n), u_n \rd\\
	& =  \left( \frac{1}{2}-\frac{1}{2\cdot 2^*_{\mu}}\right) \|u_n\|^2 -\la \left( \frac{1}{2}-\frac{1}{2\cdot 2^*_{\mu}}\right) \int_{ \Om}|u_n|^2~dx\\
	& \geq  \left( \frac{1}{2}-\frac{1}{2\cdot 2^*_{\mu}}\right) \left(1-\frac{\la}{\la_1}\right) \|u_n\|^2. 
	\end{align*}
	This proves that $u_n$ is a bounded sequence in $H_0^1(\Om)$. It implies that $\{\ld T_\la^\prime(u_n), u_n\rd  \}$ is a bounded sequence in $\mathbb{R}$ and there exists $\kappa \in (-\infty,0]$ such that, up to a subsequence, 
	$\ld T_\la^\prime(u_n), u_n\rd  \ra \kappa $ as $n \ra \infty$. Let if possible, $\kappa<0$ then using the fact that $u_n \in N_\la^\Om$ and \eqref{lu19}, we have 
	\begin{align*}
	\ld  \al_n T_\la^\prime(u_n), u_n\rd  \ra 0 \text{ as } n \ra \infty.
	\end{align*}
This implies $\al_n \ra 0$ as $n \ra \infty$. That is, 
\begin{align*}
\sup \{|\ld J_\la^\prime(u_n) , \phi \rd | : \phi \in H_0^1(\Om), \|\phi \|=1 \} \ra 0 \text{ as } n\ra \infty, 
\end{align*}
which on employing Lemma \ref{lulem2} gives that $u_n$ has a convergent subsequence. At last suppose $\kappa=0$. Since 
\begin{align*}
\ld   T_\la^\prime(u_n), u_n\rd  = \la(2-q)\int_{ \Om}|u_n|^q~dx+(2-2\cdot 2^*_{\mu}) \|u_n\|_{NL}^{2\cdot 2^*_{\mu}}\ra \kappa,  
\end{align*}
 then  $\int_{ \Om}|u_n|^q~dx \ra 0 $ and $\|u_n\|_{NL}^{2\cdot 2^*_{\mu}}\ra0$. Taking into account the fact $u_n \in N_\la^{\Om}$ we have $ \|u_n\| \ra 0$. That is, $u_n \ra 0$ strongly in $H_0^1(\Om)$. \QED
\end{proof}
In order to proceed further we will use the minimizer of $S_{H,L}$. From Lemma \ref{lulem13} we know that 
\begin{align*}
U_\e(x)= S^{\frac{(N-\mu)(2-N)}{4(N-\mu+2)}}(C(N,\mu))^{\frac{2-N}{2(N-\mu+2)}}\left(\frac{\e}{\e^2+|x|^2}\right)^{\frac{N-2}{2}}, \; 0<\e<1
\end{align*}
are the minimizers of $S_{H,L}$.  Without loss of generality, let us assume that  $0 \in \Om$.
This implies there exists a $\de>0$ such that $B_{4\de}(0)\subset \Om$.  Now  define $\eta\in C_c^{\infty}(\mathbb{R}^N)$ such that $0\leq \eta\leq 1$ in $\mathbb{R}^N$, $\eta\equiv1$ in $B_\de(0)$ and $\eta\equiv 0 $ in $\mathbb{R}^N \setminus B_{2\de}(0)$ and $|\na \eta |< C$. Let $u_\e\in  H_0^1(\Om)$ be defined as  $u_\e(x)= \eta(x) U_\e(x)$. 
\begin{Proposition}\label{luprop3}
	Let $N\geq 3,\; 0<\mu<N$ and $q\in (2,2^*)$ then the following holds:
	\begin{enumerate}
		\item [(a)] $\|u_\e\|^2 \leq  S_{H,L}^{\frac{2N-\mu}{N-\mu+2}}+O(\e^{N-2})$.
		\item [(b)] $\|u_\e\|_{NL}^{2\cdot 2^*_{\mu}}\leq S_{H,L}^{\frac{2N-\mu}{N-\mu+2}}+O(\e^N)$ and $\|u_\e\|_{NL}^{2\cdot 2^*_{\mu}}\geq S_{H,L}^{\frac{2N-\mu}{N-\mu+2}}-O(\e^N)$.
		\item [(c)]$\ds \int_{ \Om}|u_\e|^2~dx \geq C \left\{
		\begin{array}{ll}
		\e^2+ O(\e^{N-2}), &  \text{ if }  N>4, \\
		\e^{2} |\log\; \e|+O(\e^{2}), & \text{ if } N=4\\
		\e^{N-2} +O(\e^{2}), & \text{ if } N<4.\\
		\end{array} 
		\right.$ 
		\item[(d)]  $\ds \int_{ \Om}|u_\e|^q~dx \geq O(\e^{N-\frac{N-2}{2}q})$ whenever 
		$q \in (2,2^*)$ and $N>3  $ OR $4<q<6 $ and $N=3$.
	\end{enumerate}
\end{Proposition}
\begin{proof}
	For  (a) and (c) See \cite[Lemma 1.46]{willem}. For (b) See \cite[Proposition 2.8]{systemchoq}. For  (d),  first let $N>3$ and $2<q<2^*$  then $0<  (N-2)q-N<N$. Now let $N=3$ and $4<q<6$ then $1<q-3<3$. Hence we have the following estimate 
		\begin{align*}
	\int_{ \Om}|u_\e|^q~dx& \geq  C\int_{ |x|<\de}|U_\e|^q~dx\\
	& \geq C\e^{N-\frac{N-2}{2}q} \int_{1}^{\frac{\de}{\e}} r^{N-1-(N-2)q}~dx\\
	& = \frac{C\e^{N-\frac{N-2}{2}q} }{(N-2)q-N}\left[1- \left(\frac{\e}{\de}\right)^{(N-2)q-N}\right]= O(\e^{N-\frac{N-2}{2}q}).
	\end{align*}
	\QED
\end{proof}
\begin{Lemma}\label{lulem3}
	Let $N\geq 3$ and $\la>0$ and condition $\mathcal{(Q)} $ holds. Then $\widehat{\theta}_\la< \frac{N-\mu+2}{2(2N-\mu)}S_{H,L}^{\frac{2N-\mu}{N-\mu+2}}$.
\end{Lemma}
\begin{proof}
By the definition of $\widehat{\theta}_\la$, it is enough to show that for   $u_\e \in H_0^1(\Om)$, 
\begin{align*}
\ds\sup_{t\geq 0}J_\la(tu_\e)< \frac{N-\mu+2}{2(2N-\mu)}S_{H,L}^{\frac{2N-\mu}{N-\mu+2}}. 
\end{align*} 
 Let 
	\begin{align*}
	\mathcal{G}(t)= 	J_\la(t u_\e) & =  \frac{t^2}{2}\| u_\e\|^2 - \frac{\la t^q}{q} \int_{ \Om} |u_\e|^q~dx - \frac{t^{2\cdot 2^*_{\mu}}}{2\cdot 2^*_{\mu}} \|u_\e\|_{NL}^{2\cdot 2^*_{\mu}},
	\end{align*}	
	then  using the same assertions as in Lemma \ref{lulem8} for the function $\mathcal{G}$, we deduce that there exists unique  $t_\e>0$ such that $\ds \sup_{t\geq 0}\mathcal{G}(t)= \mathcal{G}(t_\e) =J_\la(t_\e u_\e)$ and $\mathcal{G}^\prime(t_\e)=0$, provided $\la\in (0,\la_1)$.  As a result, we obtain 
	\begin{align}\label{lu18}
t_\e^2 \|u_\e\|^2 -\la t_\e^q \int_{\Om}|u_\e|^q~dx- t_\e^{2\cdot 2^{*}_{\mu}}\|u_\e\|_{NL}^{2\cdot 2^{*}_{\mu}}=0.
	\end{align}
It implies $ \|u_\e\|^2= \la t_\e^{q-2}\int_{\Om}|u_\e|^q~dx+ t_\e^{2\cdot 2^{*}_{\mu}-2}\|u_\e\|_{NL}^{2\cdot 2^{*}_{\mu}}  $.  Therefore, using Proposition \ref{luprop3}, Sobolev embedding, definition of $S_{H,L}$ and the fact that $\la \in (0,\la_1)$,  we deduce  
	\begin{align*}
1& \leq \la C_1t_\e^{q-2}\|u_\e\|^{q-2}+C_2t_\e^{2\cdot 2^{*}_{\mu}-2}\|u_\e\|^{2\cdot 2^{*}_{\mu}-2}, 
	\end{align*}
	for some suitable constants $C_1, C_2>0$. It 	 gives that there exists  a $T_1>0$ such that $t_\e \geq T_1$. Also,  from \eqref{lu18}, 	 $t_\e^{2\cdot 2^{*}_{\mu}}\|u_\e\|_{NL}^{2\cdot 2^{*}_{\mu}}\leq t_\e^2 \|u_\e\|^2  $. That is, 
	\begin{align*}
	t_\e\leq \left(\frac{\| u_\e\|^2}{\| u_\e\|_{NL}^{2\cdot 2^*_{\mu}}}\right)^{\frac{1}{2\cdot 2^*_{\mu}-2}}. 
	\end{align*}
	Hence  
	\begin{align*}
	\sup_{t\geq 0}\mathcal{G}(t)& = \frac{t_\e^2}{2}\| u_\e\|^2 - \frac{\la t_\e^q}{q} \int_{ \Om} |u_\e|^q~dx - \frac{t_\e^{2\cdot 2^*_{\mu}}}{2\cdot 2^*_{\mu}} \|u_\e\|_{NL}^{2\cdot 2^*_{\mu}}\\
	& \leq  \sup_{ t\geq 0} \mathcal{V}(t) - \frac{\la T_1^q}{q} \int_{ \Om} |u_\e|^q~dx,
	\end{align*}
	where $\mathcal{V}(t)= \ds \frac{t^2}{2}\| u_\e\|^2 - \frac{t^{2\cdot 2^*_{\mu}}}{2\cdot 2^*_{\mu}} \|u_\e\|_{NL}^{2\cdot 2^*_{\mu}}$. Now using proposition \ref{luprop3} and the fact that  $\mathcal{V}(t)$ has maximum at $ t^*= \left(\frac{\| u_\e\|^2}{\| u_\e\|_{NL}^{2\cdot 2^*_{\mu}}}\right)^{\frac{1}{2\cdot 2^*_{\mu}-2}} $,  we get 
	\begin{align}\label{lu2}
	\sup_{t\geq 0}\mathcal{G}(t)\leq \frac{N-\mu+2}{2(2N-\mu)}S_{H,L}^{\frac{2N-\mu}{N-\mu+2}}+C_1\e^{N-2} -  \frac{\la T_1^q}{q} \int_{ \Om} |u_\e|^q~dx . 
	\end{align}
	\textbf{Case 1: } $N>3$ and  $q \in (2,2^*)$  OR  $N=3$ and $4<q<6$. \\
As a consequence of Proposition \ref{luprop3} and \eqref{lu2}, we have 
	\begin{align*}
	\sup_{t\geq 0}\mathcal{G}(t)&  \leq \frac{N-\mu+2}{2(2N-\mu)}S_{H,L}^{\frac{2N-\mu}{N-\mu+2}}+C_1\e^{N-2} -  \frac{\la T_1^q}{q} \int_{ \Om} |u_\e|^q~dx\\
	&  \leq \frac{N-\mu+2}{2(2N-\mu)}S_{H,L}^{\frac{2N-\mu}{N-\mu+2}}+C_1\e^{N-2} -  \frac{\la T_1^q}{q} C_2\e^{N-\frac{N-2}{2}q}.
		\end{align*}
		Now using the condition of $N$ and $q$, we have  $N-\frac{N-2}{2}q<N-2$ then for $\e$ sufficiently small,  $C_1\e^{N-2} -\frac{\la T_1^q}{q} C_2\e^{N-\frac{N-2}{2}q}<0$. Therefore,
		\begin{align*}
		\sup_{t\geq 0} J_\la(t u_\e)= \sup_{t\geq 0}\mathcal{G}(t)< \frac{N-\mu+2}{2(2N-\mu)}S_{H,L}^{\frac{2N-\mu}{N-\mu+2}}.  
		\end{align*}
	\textbf{Case 2: } If $q=2$ and $N>3$.\\
	When $N>4$ then by Proposition \ref{luprop3} and \eqref{lu2}, 
		\begin{align*}
	\sup_{t\geq 0}\mathcal{G}(t)&    \leq \frac{N-\mu+2}{2(2N-\mu)}S_{H,L}^{\frac{2N-\mu}{N-\mu+2}}+C_1\e^{N-2} -  \frac{\la T_1^2}{2} C_2\e^{2}.
	\end{align*}
	Therefore,  for $\e$ sufficiently small,  $C_1\e^{N-2} -\frac{\la T_1^2}{2} C_2\e^{2}<0$, we obtain 
	\begin{align*}
	\sup_{t\geq 0} J_\la(t u_\e)< \frac{N-\mu+2}{2(2N-\mu)}S_{H,L}^{\frac{2N-\mu}{N-\mu+2}}.  
	\end{align*}
	When $N=4$ then again  by Proposition \ref{luprop3} and \eqref{lu2}, for an appropriate constant $C_3>0$, we have  
	\begin{align*}
	\sup_{t\geq 0}\mathcal{G}(t)&    \leq \frac{N-\mu+2}{2(2N-\mu)}S_{H,L}^{\frac{2N-\mu}{N-\mu+2}}+C_1\e^{2} -  \frac{\la T_1^2}{2} C_2(\e^2 |\log \e|+\e^{2})\\
	& \leq \frac{N-\mu+2}{2(2N-\mu)}S_{H,L}^{\frac{2N-\mu}{N-\mu+2}}+C_3\e^{2} -  \frac{\la T_1^2}{2} C_2\e^2 |\log \e|.
	\end{align*}
Since  $|\log \e|\ra \infty$ as $\e \ra 0$,   for $\e$ sufficiently small,  $C_3\e^{2} -  \frac{\la T_1^2}{2} C_2\e^2 |\log \e|<0$.
	Thus
	\begin{align*}
	\sup_{t\geq 0} J_\la(t u_\e)< \frac{N-\mu+2}{2(2N-\mu)}S_{H,L}^{\frac{2N-\mu}{N-\mu+2}}.  
	\end{align*}

	\QED
\end{proof}

\begin{Lemma}\label{lulem4}
	If condition $\mathcal{(Q)}$ holds then  the following holds. 
	\begin{enumerate}
		\item [(a)]  $\widehat{\theta}_\la= \theta_\la$.
		\item [(b)] $0<\theta_\la< \frac{N-\mu+2}{2(2N-\mu)}S_{H,L}^{\frac{2N-\mu}{N-\mu+2}}$. 
		\item [(c)] 	There exists $u_\la^\Om \in N_\la^{\Om}$ such that $J_\la(u_\la^\Om) = \ds \inf_{u \in N_\la^{\Om}}J_\la(u) = \theta_\la$ and $u_\la^\Om \geq 0$.
	\end{enumerate} 
\end{Lemma}
\begin{proof}
\begin{enumerate}
	\item [(a)]By Lemma \ref{lulem2}, Lemma \ref{lulem3}, Lemma \ref{efflem1} and  Mountain Pass Lemma, there exists a $u_\la^\Om \in H_0^1(\Om)$ such that $J_\la(u_\la^\Om)= \widehat{\theta}_\la$ and $J_\la^\prime(u_\la^\Om)=0$. It implies $u_\la^\Om \in N_\la^\Om$. Hence, $\theta_\la \leq J_\la(u_\la^\Om)= \widehat{\theta}_\la$. Also  from Lemma \ref{lulem8}, for   each  $v \in N_\la^\Om$, there exists a unique $t_0>0$ such that    $\ds \sup_{ t\geq 0}J_\la(tv)= J_\la(t_0v)$. Since $u_\la^\Om \in N_\la^\Om$, it implies  $\widehat{\theta}_\la \leq \ds \sup_{ t\geq 0}J_\la(tu)= J_\la(u)$. Therefore, $\widehat{\theta}_\la \leq \theta_\la$. 
	\item [(b)]By Lemma \ref{lulem8},  $\theta_\la>0$ and by Lemma \ref{lulem3}, $\theta_\la = \widehat{\theta}_\la < \frac{N-\mu+2}{2(2N-\mu)}S_{H,L}^{\frac{2N-\mu}{N-\mu+2}}$. 
	\item [(c)] By part (a), there exists a $u_\la^\Om \in N_\la^{\Om}$ such that $  J_\la(u_\la^\Om)= \widehat{\theta}_\la= \theta_\la=\ds \inf_{u \in N_\la^{\Om}}J_\la(u)$. Since $ J_\la(u_\la^\Om)=  J_\la(|u_\la^\Om|)$, we can assume $u_\la^\Om\geq0$.  \QED
\end{enumerate} 
	\end{proof}
\section{Proof of Theorem \ref{luthm1}}
In this section, first we  gather  some information which is needed to estimate the $\text{cat}_\Om(\Om)$. Before that, we prove some Lemmas which are necessary for the proof of Theorem \ref{luthm1}.  	

\begin{Lemma}\label{lulem7}
	Let $N\geq 3$ and $\{ u_n\}$ be a sequence in $H_0^1(\Om)$ such that
	\begin{align*}
	\|u_n\|_{NL}^{2\cdot 2^{*}_{\mu}}= 
	\|u_n\|^2\leq  S_{H,L}^{\frac{2N-\mu}{N-\mu+2}}+o_n(1) \text{ as } n \ra  \infty.\end{align*}
	Then, there exist  sequences $z_n\in \mathbb{R}^N$ and $\al_n\in \mathbb{R}^+$ such that  the sequence 
	\begin{align*}
	v_n (x)= \al_n^{\frac{N-2}{2}} u_n(\al_nx+z_n) 
	\end{align*}
	have a convergent subsequence, still denoted by $v_n$. Moreover, $v_n \ra v \not \equiv 0 $ in $D^{1,2}(\mathbb{R}^N),\; z_n \ra z \in \overline{\Om}$ and  $\al_n  \ra 0$ as $n \ra \infty$. 
\end{Lemma}
\begin{proof}
	Let  $\{ w_n\}$ be a sequence such that $w_n= \ds \frac{u_n}{\|u_n\|_{NL}}$ then $\|w_n\|_{NL}=1, \; \|w_n\|^2=  \ds \frac{\|u_n\|^2}{\|u_n\|_{NL}^2}= \|u_n\|^{2(\frac{N-\mu+2}{2N-\mu})}\leq S_{H,L}+o_n(1)$. By definition of $S_{H,L}$, $ \|w_n\|^2\geq S_{H,L} $, it implies $\|w_n\|^2 \ra S_{H,L}$ as $n\ra \infty$. Now using Proposition \ref{luprop2} for the sequence $\{w_n\}$, we have the desired result.\QED
\end{proof}
Since $\Om$ is a smooth bounded domain of $\mathbb{R}^N$, thus we can pick $\de>0$ small enough so that 
\begin{align*}
\Om_\de^+=\{x \in \mathbb{R}^N \;|\;  \text{dist}(x,\Om)< \de \} \quad \text{and} \quad \Om_\de^-=\{x \in \mathbb{R}^N \;|\;  \text{dist}(x,\Om)> \de \}
\end{align*}
are homotopically equivalent to $\Om$. Without loss of generality, we can assume that $B_\de= B_\de(0) \subset \Om$. Consequently, we consider the functional $J_{\la}^{B_\de}: H_{0,\text{rad}}^1(B_\de)\ra \mathbb{R}^N$ defined as 
\begin{align*}
J_{\la}^{B_\de}(u)= \frac{1}{2}\int_{B_\de} |\na u|^2~dx -\frac{\la}{q}\int_{B_\de}|u|^q- \frac{1}{2\cdot 2^{*}_{\mu}}\int_{B_\de}\int_{B_\de}\frac{|u(x)|^{2^{*}_{\mu}}|u(y)|^{2^{*}_{\mu}}}{|x-y|^{\mu}}~dydx, 
\end{align*}
where $H_{0,\text{rad}}^1(B_\de)= \{ u \in H_{0}^1(B_\de): u \text{ is radial} \}.$
And let $N_\la^{B_\de}$  be the Nehari manifold associated to  functional $J_{\la}^{B_\de}$. Then all the results obtained in Section 3  are valid for the functional $J_{\la}^{B_\de}$. In particular, by Lemma  \ref{lulem4}, we know that there exists $u_\la^{B_\de} \in N_\la^{B_\de}$ such that $u_\la^{B_\de}\geq 0$ in $B_\de$. Moreover,
\begin{align}\label{lu5}
J_{\la}^{B_\de}(u_\la^{B_\de} ) =\inf_{u \in N_\la^{B_\de}}J_{\la}^{B_\de}(u )< \frac{N-\mu+2}{2(2N-\mu)}S_{H,L}^{\frac{2N-\mu}{N-\mu+2}}.
\end{align}
 Now with the help of $u_\la^{B_\de}$ we will define the following set 
 \begin{align*}
 \mathcal{A}_\la= \{ u \in N_\la^\Om : J_\la(u)\leq  J_{\la}^{B_\de}(u_\la^{B_\de} ) \}, 
 \end{align*} 
 and the function $\phi_\la: \Om_\de^- \ra \mathcal{A}_\la$ given by 
\begin{align}\label{lu6}
\left\{
\begin{array}{ll}
u_\la^{B_\de}(x-z),  &  \text{ if } x \in B_\de(z), \\
0, &  \text{ elsewhere }. \\
\end{array} 
\right.
 \end{align}	
In the succession, we define the barycenter mapping $\beta:N_\la^{\Om} \ra \mathbb{R}^N$ by setting 
\begin{align}\label{lu16}
\beta(u)=  \frac{\ds \int_{\Om}\int_{\Om}\frac{x|u(x)|^{2^{*}_{\mu}}|u(y)|^{2^{*}_{\mu}}}{|x-y|^{\mu}}~dydx}{\ds \|u\|_{NL}^{2\cdot 2^*_{\mu}} }
\end{align}
Using the fact that  $u_\la^{B_\de} $ is radial,  $\beta(\phi_\la(z))= z$ for all $z \in \Om_\de^-$.

\begin{Lemma}\label{lulem9}
	Let $N\geq 3$ and $q \in [2,2^*)$. Then there exists $\Upsilon^*>0 $ such that  if $u \in \mathcal{A}_\la$ and $ \la \in (0,\Upsilon^*)$ then $\beta(u) \in \Om_\de^+$.
\end{Lemma}
\begin{proof}
	On the contrary, let there exists  sequences $\{\la_n\} \in \mathbb{R}^+$  and $u_n \in \mathcal{A}_{\la_n}$ such that  $ \la_n  \ra 0$  and $\beta(u_n) \not \in  \Om_\de^+$.  Using the definition of $\mathcal{A}_{\la_n}$, we have $u_n \in N_{\la_n}^\Om$  and $ J_{\la_n}(u_n)\leq  J_{\la_n}^{B_\de}(u_{\la_n}^{B_\de} )$. Define
\begin{align*}
M(t)= 	J_{\la_n}(tu_n) =  \frac{t^2}{2} \|u_n\|^2 -\frac{\la_n t^q}{q} \int_{\Om}|u_n|^q~dx-  \frac{t^{2\cdot 2^{*}_{\mu}}}{2\cdot 2^{*}_{\mu}}\|u_n\|_{NL}^{2\cdot 2^{*}_{\mu}},
\end{align*}
using the same assertions and arguments as in Lemma \ref{lulem8}, there exists a unique $t_0>0$ such that $M^\prime(t_0)=0$ and $t_0u_n \in N_{\la_n}^\Om$. Since $u_n \in N_{\la_n}^\Om$, it implies that  $M^\prime(1)=0$ and $M$ is increasing for $t<1$ and decreasing $t>1$. Therefore, 
\begin{align}\label{lu7}
J_{\la_n}(u_n)= \sup_{ t\geq 0} J_{\la_n}(tu_n) .
\end{align}
As $ \|u_n\|^2 -\la_n \int_{\Om}|u_n|^q~dx-  \|u_n\|_{NL}^{2\cdot 2^{*}_{\mu}}=0$, employing this with definition of $S_{H,L}$ and Sobolev embedding, we have 
\begin{align*}
1= & \frac{\la_n \int_{\Om}|u_n|^q~dx}{\|u_n\|^2}+ \frac{\|u_n\|_{NL}^{2\cdot 2^{*}_{\mu}}}{\|u_n\|^2} \leq \la_n  c_1\|u_n\|^{q-2}+S_{H,L}^{-2^*_{\mu}}\|u_n\|^{2\cdot 2^{*}_{\mu}-2},
\end{align*}
where $c_1>0$ is a appropriate  constant. It implies that for large $n$, there exists a constant $C>0$ such that 
\begin{align}\label{lu8}
\|u_n\|>C. 
\end{align}
\textbf{Claim 1: } There exists a $l>0$ such that up to a subsequence $\|u_n\|_{NL}^{2\cdot 2^{*}_{\mu}} \ra l$ as $n\ra \infty$. \\
Since  $J_{\la_n}(u_n) \leq  J_{\la_n}^{B_\de}(u_{\la_n}^{B_\de} )< \frac{N-\mu+2}{2(2N-\mu)}S_{H,L}^{\frac{2N-\mu}{N-\mu+2}}$, $J_{\la_n}(u_n)$ is bounded in $\mathbb{R}$, subsequently $\|u_n\|_{NL}$ is a bounded sequence. Moreover, from the fact that $u_n \in N_{\la_n}^\Om$, it follows that
\begin{align*}
J_{\la_n}(u_n) = \la_n \left( \frac{1}{2}-\frac{1}{q} \right)\int_{\Om}|u_n|^q~dx+ \left( \frac{1}{2}-\frac{1}{2\cdot 2^{*}_{\mu}} \right)\|u_n\|_{NL}^{2\cdot 2^{*}_{\mu}}, 
\end{align*}
It implies that  $\la_n \int_{\Om}|u_n|^q~dx$ is a bounded sequence. As a consequence, $\|u_n\|$ is bounded in $\mathbb{R}$. Therefore, there exists a $l\geq 0$  such that $\|u_n\|_{NL} \ra l$ as $n \ra \infty$. To prove the Claim 1, it is enough to show that $l\not=0$. Using \eqref{lu8}, we deduce 
\begin{align*}
\|u_n\|_{NL}^{2\cdot 2^{*}_{\mu}} & = \|u_n\|^2- \la_n \int_{\Om}|u_n|^q~dx \geq \|u_n\|^2- \la_n c_1\|u_n\|^q \geq C^2- \la_n c_2, 
\end{align*}
where $c_2>0$ is a suitable constant. Since $ \la_n \ra 0$, so we have $l>0$. This proves Claim 1.\\
\textbf{Claim 2:} For all $n\in \mathbb{N}$, there exists $t_n>0$ such that $\|t_n u_n\|^2 = \|t_n u_n\|_{NL}^{2\cdot 2^{*}_{\mu}}$. Furthermore, $t_n$ is a bounded sequence in $\mathbb{R}$. \\
Assume $t_n = \left[\frac{\|u_n\|^2}{\|u_n\|_{NL}^{2\cdot 2^{*}_{\mu}}} \right]^{\frac{1}{2\cdot 2^*_{\mu}-2}}$ then $\|t_n u_n\|^2 = \|t_n u_n\|_{NL}^{2\cdot 2^{*}_{\mu}}  \text{ for all } n \in \mathbb{N}$. Using the fact that $\|u_n\|$ is bounded and by Claim 1, we deduce that  $t_n$ is a bounded sequence in $\mathbb{R}$, concludes the proof of Claim 2. \\
By the definition of $J_{\la_n}$ and  taking into account  \eqref{lu5}, \eqref{lu7}, Claim 2, $u_n\in \mathcal{A}_{\la_n}$, $\la_n \ra 0$, and  $\int_{\Om}|u_n|^q~dx $ is bounded, we obtain
\begin{align*}
\frac{N-\mu+2}{2(2N-\mu)} \|t_n u_n\|^2 & = J_{\la_n}(t_nu_n)+ \la_n t_n^q \int_{\Om}|u_n|^q~dx \\
& \leq J_{\la_n}(u_n)+ o_n(1)\\
&\leq  J_{\la_n}^{B_\de}(u_{\la_n}^{B_\de} )+ o_n(1) < \frac{N-\mu+2}{2(2N-\mu)}S_{H,L}^{\frac{2N-\mu}{N-\mu+2}}+ o_n(1). 
\end{align*}
From Claim 2  and Lemma \ref{lulem7}, there exists a  sequences $z_n\in \mathbb{R}^N$ and $\al_n\in \mathbb{R}^+$ such that  the sequence 
\begin{align*}
v_n (x)= \al_n^{\frac{N-2}{2}} t_n u_n(\al_nx+z_n) 
\end{align*}
have a convergent subsequence, still denoted by $v_n$. Moreover, $v_n \ra v \not \equiv 0 $ in $D^{1,2}(\mathbb{R}^N),\; z_n \ra z \in \overline{\Om}$ and  $\al_n  \ra 0$ as $n \ra \infty$. 
Let $\psi \in C_c^{\infty}(\mathbb{R}^N)$ such that $\psi(x)=x $ for all $x \in \overline{\Om}$.  Consider 
\begin{align*}
\beta(u_n)= \beta(t_nu_n) =&  \frac{\ds \int_{\mathbb{R}^N}\int_{\mathbb{R}^N}\frac{\psi(x)|u_n(x)|^{2^{*}_{\mu}}|u_n(y)|^{2^{*}_{\mu}}}{|x-y|^{\mu}}~dydx}{\ds \int_{\mathbb{R}^N}\int_{\mathbb{R}^N}\frac{|u_n(x)|^{2^{*}_{\mu}}|u_n(y)|^{2^{*}_{\mu}}}{|x-y|^{\mu}}~dydx }\\
& = \frac{\ds \int_{\mathbb{R}^N}\int_{\mathbb{R}^N}\frac{\psi(\al_nx+z_n)|v_n(x)|^{2^{*}_{\mu}}|v_n(y)|^{2^{*}_{\mu}}}{|x-y|^{\mu}}~dydx}{\ds \int_{\mathbb{R}^N}\int_{\mathbb{R}^N}\frac{|v_n(x)|^{2^{*}_{\mu}}|v_n(y)|^{2^{*}_{\mu}}}{|x-y|^{\mu}}~dydx }\\
& \ra z \in \overline{\Om},
\end{align*}
where the last one follows from regularity of $\psi$ and Lebesgue dominated theorem. This  contradicts the assumption $\beta(u_n) \not \in  \Om_\de^+$. It concludes the proof. \QED
\end{proof}
\begin{Lemma}\label{lulem10}
Assume  $N\geq 3,\; q \in [2,2^*)$ and $\la\in (0,\Upsilon^*)$ (defined in Lemma \ref{lulem9}). Then $\text{cat}_{\mathcal{A}_\la}(\mathcal{A}_\la)\geq \text{cat}_\Om(\Om)$
\end{Lemma}
\begin{proof}
	The proof can be done by using the same assertions  as in \cite[Lemma 4.3]{alves}.\QED		
		\end{proof}
	Next we  need following  lemma in order to proof Theorem \ref{luthm1}. 
	\begin{Lemma}\label{lulem11}\cite{ambrosetti}
		Suppose that $X$ is a Hilbert manifold and $F \in C^1(X,\mathbb{R})$ . Assume that there are $c_1 \in \mathbb{R}$ and  $k \in  \mathbb{N}$,
		such that
		\begin{enumerate}
			\item $F$ satisfies the Palais-Smale condition for energy level $c\leq c_1$;
			\item  $\text{Cat}(\{x\in X\; |\; F(x)\leq c_1 \})\geq k$.
		\end{enumerate}
		Then $F$ has at least $k$ critical points in  $\{x\in X\; |\; F(x)\leq c_1 \}$.
	\end{Lemma}

\noi \textbf{Proof of Theorem \ref{luthm1} :} By Lemma \ref{lulem14}, $J_\la$ satisfies $(PS)_c$ condition on $N_\la^{\Om}$ for any $c<\frac{N-\mu+2}{2(2N-\mu)}S_{H,L}^{\frac{2N-\mu}{N-\mu+2}}$, provided $\la\in (0,\la_1)$. If condition $\mathcal{(Q)}$ holds then  from Lemma \ref{lulem4},  $0<\theta_\la< \frac{N-\mu+2}{2(2N-\mu)}S_{H,L}^{\frac{2N-\mu}{N-\mu+2}}$. Hence if condition $\mathcal{(Q)}$ holds then Lemmas \ref{lulem10} and \ref{lulem11}, we have at least $\text{cat}_\Om(\Om)$ critical points of $J_\la$ restricted to $N_\la$ for any $\la\in (0,\La^*)$, where
\begin{align*}
\La^*=  \min\{\la_1,\Upsilon^* \}, 
\end{align*}
Thus using Lemma \ref{lulem12}, we obtain $J_\la$ has at least $\text{cat}_\Om(\Om)$ critical points on $H_0^1(\Om)$. From \cite[Lemma 4.4]{yangjmaa}, we have at least $\text{cat}_\Om(\Om)$  positive solutions of problem $(P_\la)$.\QED

\appendix

\noindent

\section{\!\!\!\!\!\!ppendix}
\label{A}
Here we will proof behavior of the optimizing sequence of $S_{H,L}$.  For the local case, Proposition \ref{luprop2} has been proved in \cite{struwe2} and \cite{willem}. Combining
the ideas of  \cite{gaoyang} and \cite{willem},  one expects the Proposition \ref{luprop2}  to hold for critical Choquard case, but as best of our knowledge this type of result  has not been
proved exclusively anywhere. For $N=3$, Proposition \ref{luprop2} has been proved in \cite{Sililiano}.
\begin{Proposition}\label{luprop2}
	Let $\{u_n\}$ be a sequence in $H_0^1(\Om)$ such that
	\begin{align*}
	\int_{\Om}\int_{\Om}\frac{|u_n(x)|^{2^{*}_{\mu}}|u_n(y)|^{2^{*}_{\mu}}}{|x-y|^{\mu}}~dydx=1 \text{ and  } \|u_n\|^2\ra S_{H,L} \text{ as } n \ra  \infty.\end{align*}
	Then, there exists a  sequences $z_n\in \mathbb{R}^N$ and $\al_n\in \mathbb{R}^+$ such that  the sequence 
	\begin{align*}
	v_n (x)= \al_n^{\frac{N-2}{2}} u_n(\al_nx+z_n) 
	\end{align*}
	have a convergent subsequence, still denoted by $v_n$, such that $v_n \ra v \not \equiv 0 $ in $D^{1,2}(\mathbb{R}^N),\; z_n \ra z \in \overline{\Om}$, and  $\al_n  \ra 0$ as $n \ra \infty$. In particular, $v$ is a minimizer of $S_{H,L}$.  
\end{Proposition}
\begin{proof}
	Define the L\'evy concentration function 
	\begin{align*}
	Q_n(\la):= \sup_{z \in \mathbb{R}^N} \int_{B(z,\la)}(|x|^{-\mu}* |u_n|^{2^*_{\mu}}) |u_n|^{2^*_{\mu}}~dx . 
	\end{align*}
	It is easy to see that  for each $n, \; \ds \lim_{\la \ra 0^+}	Q_n(\la)=0$ and $\ds \lim_{\la \ra \infty}Q_n(\la)=1$, there exists $\al_n>0$ such that $Q_n(\al_n)=\frac{1}{2}$. Also, there exist $z_n\in \mathbb{R}^N$ such that 
	\begin{align*}
	\int_{B(z_n,\al_n)}(|x|^{-\mu}* |u_n|^{2^*_{\mu}}) |u_n|^{2^*_{\mu}}~dx = Q_n(\al_n)=\frac{1}{2}.
	\end{align*}
	Now define the function $v_n(x)= \al_n^{\frac{N-2}{2}} u_n(\al_nx+z_n)$ then 
	\begin{align}\label{nh22}
	& \int_{\mathbb{R}^N}\int_{\mathbb{R}^N}\frac{|v_n(x)|^{2^{*}_{\mu}}|v_n(y)|^{2^{*}_{\mu}}}{|x-y|^{\mu}}~dxdy=1,\;   \|\na v_n\|^2_{L^2}\ra S_{H,L} \text{ as } n \ra  \infty \text{ and  } \nonumber \\
	& \frac{1}{2}= \sup_{z \in \mathbb{R}^N} \int_{B(z,1)}(|x|^{-\mu}* |v_n|^{2^*_{\mu}}) |v_n|^{2^*_{\mu}}~dx=  \int_{B(0,1)}(|x|^{-\mu}* |v_n|^{2^*_{\mu}}) |v_n|^{2^*_{\mu}}~dx. 
	\end{align}
	It implies $\{v_n\} $ is a bounded sequence in $D^{1,2}(\mathbb{R}^N)$. Therefore, there exist a  subsequence, still denoted by  $\{v_n\}$ such that $v_n \rp v $  weakly in $ D^{1,2}(\mathbb{R}^N)$, for some   $v\in D^{1,2}(\mathbb{R}^N)$.  Then we can  assume that  there exist $\omega,\; \tau,\; \nu$ such that
	\begin{align*}
	v_n \ra v \text{ a.e on }  \mathbb{R}^N , \; 	|\na v_n|^2 \rp \omega,\; |v_n|^{2^*}\rp \tau ,\text{ and } (|x|^{-\mu}* |v_n|^{2^*_{\mu}}) |v_n|^{2^*_{\mu}} \rp \nu \text{ in the sense of measure}. 
	\end{align*}
	Now using the Brezis-Leib lemma in sense of measure, we have
	\begin{align*}
	& |\na (v_n-v)|^2 \rp \varpi:= \omega-|\na v|^2,\; |v_n-v|^{2^*}\rp \chi:= \tau- |v|^{2^*} ,\text{ and }\\
	&  (|x|^{-\mu}* |v_n-v|^{2^*_{\mu}}) |v_n-v|^{2^*_{\mu}} \rp\kappa:=  \nu - (|x|^{-\mu}* |v|^{2^*_{\mu}}) |v|^{2^*_{\mu}} . 
	\end{align*}
	Moreover, if we define 
	\begin{align*}
	& 	\omega_{\infty}:= \lim_{R\ra\infty}\limsup_{n\ra \infty} \int_{|x|>R}|\na v_n|^2~dx ,\\ & 
	\tau_{\infty}:= \lim_{R\ra\infty}\limsup_{n\ra \infty} \int_{|x|>R}|v_n|^{2^*}~dx,  \quad \text{ and } \\ &
	\nu_{\infty}:= \lim_{R\ra\infty}\limsup_{n\ra \infty}\int_{|x|>R}(|x|^{-\mu}* |v_n|^{2^*_{\mu}}) |v_n|^{2^*_{\mu}} ~dx  
	\end{align*}
	then by using concentration-compactness principle \cite[Lemma 2.5]{gaoyang}, we deduce that 
	\begin{align*}
	& \limsup_{n\ra \infty}\|\na v_n\|^2_{L^2}= \int_{\mathbb{R}^N} d \omega+ \omega_{\infty},\quad \; \limsup_{n\ra \infty}\|v_n\|^{2^*}_{L^{2^*}}= \int_{\mathbb{R}^N} d \tau+ \tau_{\infty},\\
	& \hspace{3cm}\limsup_{n\ra \infty}\|v_n\|^{2\cdot 2^*_{\mu}}_{NL}= \int_{\mathbb{R}^N} d \nu+ \nu_{\infty} \text{ and }\\
	& C(N, \mu)^{-\frac{2N}{2N-\mu}}\nu_{\infty}^{\frac{2N}{2N-\mu}} \leq \tau_{\infty}\left(\int_{\mathbb{R}^N} d \tau+ \tau_{\infty} \right),\;  S_{H,L}^2 \nu_{\infty}^{\frac{2}{2^*_{\mu}}}\leq \omega_{\infty}\left(\int_{\mathbb{R}^N} d \omega+ \omega_{\infty} \right).
	\end{align*}
	Also, if $v=0$ and $\ds\int_{\mathbb{R}^N} d \omega= S_{H,L}\left(\ds\int_{\mathbb{R}^N} d \nu\right)^{\frac{1}{2^*_{\mu}}}$ then $\nu$ is concentrated at a single point. By using  \cite[(2.11)]{gaoyang},  we have 
	\begin{align}\label{lu9}
	S_{H,L}\left(\int_{\mathbb{R}^N} d \kappa \right)^{\frac{1}{2^*_{\mu}}} \leq \int_{\mathbb{R}^N} d \varpi.
	\end{align} 
	It implies 
	\begin{equation}
	\begin{aligned}\label{lu10}
	& S_{H,L}=\limsup_{n\ra \infty}\|\na v_n\|^2_{L^2}= \int_{\mathbb{R}^N} d \varpi + \|\na v\|^2_{L^2}+ \omega_{\infty} ,\; \\
	& 1= \limsup_{n\ra \infty}\|v_n\|^{2\cdot 2^*_{\mu}}_{NL}= \int_{\mathbb{R}^N} d \kappa + \|v\|^{2\cdot 2^*_{\mu}}_{NL}+\nu_{\infty} \\
	& S_{H,L} \nu_{\infty}^{\frac{2}{2^*_{\mu}}}\leq \omega_{\infty}. 
	\end{aligned}
	\end{equation}
	Using the definition of $S_{H,L}$, \eqref{lu9} and \eqref{lu10}, we obtain
	\begin{align*}
	& S_{H,L}\geq S_{H,L}\left( \left(\|v\|^{2\cdot 2^*_{\mu}}_{NL}\right)^{\frac{1}{2^*_{\mu}}} + \left(\int_{\mathbb{R}^N} d \kappa \right)^{\frac{1}{2^*_{\mu}}} +\nu_{\infty}^{\frac{2}{2^*_{\mu}}} \right), \text{ that is, }\\
	& \int_{\mathbb{R}^N} d \kappa + \|v\|^{2\cdot 2^*_{\mu}}_{NL}+\nu_{\infty}  \geq  \left(\|v\|^{2\cdot 2^*_{\mu}}_{NL}\right)^{\frac{1}{2^*_{\mu}}} + \left(\int_{\mathbb{R}^N} d \kappa \right)^{\frac{1}{2^*_{\mu}}} +\nu_{\infty}^{\frac{2}{2^*_{\mu}}} 
	\end{align*}
	Thanks to the fact that  $\|v\|_{NL},\; \ds \int_{\mathbb{R}^N} d \kappa,\; \nu_{\infty} $ are  non-negative, we get  $\|v\|_{NL},\; \ds \int_{\mathbb{R}^N} d \kappa,\; \nu_{\infty} $ are equal to either 1 or 0.  Using \eqref{nh22}, we have $\nu_{\infty}\leq \frac{1}{2}$. It implies $\nu_{\infty}=0$. Now if $\ds \int_{\mathbb{R}^N} d \kappa=1$ then $\|v\|_{NL}=0$ that is, $v=0$ a.e. on $\mathbb{R}^N$. Therefore, $S_{H,L}= \ds \int_{\mathbb{R}^N} d \varpi + \omega_{\infty}\geq \ds \int_{\mathbb{R}^N} d \varpi$. Hence,
	\begin{align}\label{lu11}
	S_{H,L}\left(\int_{\mathbb{R}^N} d \kappa \right)^{\frac{1}{2^*_{\mu}}} \geq \int_{\mathbb{R}^N} d \varpi.
	\end{align} 
	Coupling \eqref{lu9}, \eqref{lu11} with the fact that $v=0$ a.e on $\mathbb{R}^N$,  we have $\nu $ is concentrated at a single point $z_0$. From \eqref{nh22}, we get 
	\begin{align*}
	\frac{1}{2}= \sup_{z \in \mathbb{R}^N} \int_{B(z,1)}(|x|^{-\mu}* |v_n|^{2^*_{\mu}}) |v_n|^{2^*_{\mu}}~dx \geq \int_{B(z_0,1)}(|x|^{-\mu}* |v_n|^{2^*_{\mu}}) |v_n|^{2^*_{\mu}}~dx\ra \int_{\mathbb{R}^N} d \kappa =1,
	\end{align*}	
	which is not possible. Hence, $\|v\|^{2\cdot 2^*_{\mu}}_{NL}=1$. Also, $S_{H,L}=\ds\lim_{n\ra \infty}\|\na v_n\|^2_{L^2}=\|\na v\|^2_{L^2}$. In particular, $v$ is a minimizer of $S_{H,L}$. From  \cite[Lemma 1.2]{yang}, we know  $S_{H,L}$  is achieved if  and only if 
	\begin{align*}
	u=C\left(\frac{b}{b^2+|x-a|^2}\right)^{\frac{N-2}{2}}
	\end{align*} 
	where $C>0$ is a fixed constant, $a\in \mathbb{R}^N$ and $b\in (0,\infty)$ are parameters. It implies $v= u=C\left(\frac{b}{b^2+|x-a|^2}\right)^{\frac{N-2}{2}}$. In particular, $v\not \equiv 0$.  Now, we will prove that $\al_n \ra 0$ and $z_n \ra z_0 \in \overline{\Om}$. Let if possible $\al_n \ra \infty $. Since $\{u_n\}$ is a bounded sequence in $H_0^1(\Om)$,   $\{u_n\}$ is a bounded sequence in $L^2(\Om)$. Thus if we define $\Om_n= \ds \frac{\Om- z_n}{\al_n}$ then 
	\begin{align*}
	\int_{\Om_n} |v_n|^2~dx = \frac{1}{\al_n^2}\int_{\Om} |u_n|^2~dx \leq \frac{C}{\al_n^2}\ra 0.
	\end{align*}
	Contrary to this, by Fatou's Lemma we have $0= \ds \liminf_{n\ra \infty}\int_{\Om_n}|v_n|^2~dx \geq \int_{\Om_n} |v|^2~dx$. This means $v\equiv0$, which is not true. Hence $\{\al_n\}$ is bounded in $\mathbb{R}$ that is, there exists $\al_0 \in \mathbb{R}$ such that $\al_n\ra \al_0$ as $n\ra \infty$. If $z_n \ra  \infty$ then for  any $x \in \Om$	and  large $n$, $\al_nx+z_n \not \in \overline{\Om}$. Since $u_n\in H_0^1(\Om)$ then $u_n(\al_nx+z_n) =0$ for all $x\in \Om$, it yields a contradiction to the assumption 
	$\|u_n\|_{NL}^{2\cdot 2^*_{\mu}}=1$. Therefore, $z_n$ is bounded, it implies that $z_n\ra z_0$. Now suppose $\al_n \ra \al_0>0$ then $\Om_n\ra  \ds \frac{\Om- z_0}{\al_0} = \Om_0 \not = \mathbb{R}^N$. Hence 
	\begin{align*}
	\int_{\Om_0}\int_{\Om_0}\frac{|v_n(x)|^{2^{*}_{\mu}}|v_n(y)|^{2^{*}_{\mu}}}{|x-y|^{\mu}}~dxdy=1 \text{ and  } \int_{\Om_0}|v_n|^2~dx \ra \int_{\Om_0}|v|^2~dx= S_{H,L} \text{ as } n \ra  \infty.\end{align*}
	which is not true. Hence $\al_n \ra 0 $ as $n\ra \infty$. Finally, arguing by contradiction, we assume that
	\begin{align}\label{lu19}
	z_0 \not \in \overline{\Om} .
	\end{align}
	In view of the fact that  $\al_nx+z_n \ra z_0$ for all $x \in \Om$  as $n \ra \infty$. Now using \eqref{lu19} we have $\al_nx+z_n  \not \in \overline{\Om}$ for all $ x \in \Om $ and $n$ large enough. It implies that $u_n(\al_nx +z_n)=0 $ for $n$ large enough. This yields a  contradiction, therefore, $ z_0\in \overline{\Om}$.\QED
\end{proof}

\noi \textbf{Acknowledgment}
The author would like to thank Prof. K Sreenadh   for various discussion that greatly improved the manuscript. The author would like to thank the anonymous referee for valuable comments.

\end{document}